\documentclass[a4paper]{article}
\usepackage{mathptmx}
\usepackage{amscd,amssymb,amsmath,amsthm}

\makeatletter

\renewcommand{\@seccntformat}[1]
{{\csname the#1\endcsname}.\hspace{0.3em}}

\renewcommand{\section}{\@startsection
{section}
{1}
{0mm}
{-1.5\baselineskip}
{\baselineskip}
{\bfseries\normalsize}}

\renewcommand{\subsection}{\@startsection
{subsection}
{2}
{0mm}
{-\baselineskip}
{0.5\baselineskip}
{\normalsize\itshape}}

\renewcommand{\subsubsection}{\@startsection
{subsubsection}
{3}
{0mm}
{-.5\baselineskip}
{-2mm}
{\normalsize\itshape}}

\makeatother

\theoremstyle{plain}
\newtheorem*{theorem*}{Theorem}
\newtheorem{theorem}{Theorem}[section]
\newtheorem{lemma}{Lemma}[section]
\newtheorem{corollary}[theorem]{Corollary}
\newtheorem{prop}[lemma]{Proposition}

\newtheorem*{corollary*}{Corollary}

\newtheorem{oquest}{Open Question}

\theoremstyle{definition}
\newtheorem*{defin*}{Definition}
\newtheorem{defin}{Definition}[section]

\theoremstyle{remark}

\newtheorem{example}[lemma]{Example}
\newtheorem*{quest*}{Question}

\DeclareMathAlphabet{\matheur}{U}{eur}{m}{n}
\DeclareMathAlphabet{\matheus}{U}{eus}{m}{n}
\DeclareMathAlphabet{\matheuf}{U}{euf}{m}{n}

\numberwithin{equation}{section}

\newcommand{\abs}[1]{\left\lvert#1\right\rvert}

\DeclareMathOperator{\dist}{dist}
\DeclareMathOperator{\card}{card}
\DeclareMathOperator{\grad}{grad}

\DeclareMathOperator{\Cut}{Cut}
\DeclareMathOperator{\Gl}{Gl}
\DeclareMathOperator{\SU}{SU}
\DeclareMathOperator{\Gr}{Gr}
\DeclareMathOperator{\Ann}{Ann}

\begin{document}

\author{Gerasim  Kokarev
\\ {\small\it School of Mathematics, The University of Leeds}
\\ {\small\it Leeds, LS2 9JT, United Kingdom}
\\ {\small\it Email: {\tt G.Kokarev@leeds.ac.uk}}
}

\title{Bounds for Laplace eigenvalues of K\"ahler metrics}
\date{}
\maketitle

\begin{abstract}
\noindent
We prove inequalities for Laplace eigenvalues of K\"ahler manifolds generalising to higher eigenvalues the classical inequality for the first Laplace eigenvalue due to Bourguignon, Li, and Yau in 1994. We also obtain similar eigenvalue inequalities for analytic varieties in K\"ahler manifolds.

\end{abstract}

\medskip
\noindent
{\small

}

%
%
%


\section{Statements and discussion of results}
\label{intro}

\subsection{Introduction}
Let $(\Sigma^n,g,J)$ be a closed K\"ahler manifold of complex dimension $n\geqslant 1$, and $\omega_g$ be its K\"ahler form. By $\Delta_g$ we denote the Laplace-Beltrami operator acting on functions on $(\Sigma^n,g)$. In 1994 Bourguignon, Li, and Yau~\cite{BLY} proved the following inequality for the first non-zero Laplace eigenvalue $\lambda_1(\Sigma^n,g)$ for projective manifolds $\Sigma^n$.
\begin{theorem}
\label{t:bly}
Let $(\Sigma^n,J)$ be an $n$-dimensional closed complex manifold that admits a holomorphic immersion $\phi:\Sigma^n\to\mathbb CP^m$. Suppose that $\Sigma^n$ is full in the sense that the image $\phi(\Sigma^n)$ is not contained in any hyperplane of $\mathbb CP^m$. Then for any K\"ahler metric $g$ on $\Sigma^n$ the first non-zero Laplace eigenvalue $\lambda_1(\Sigma^n,g)$ satisfies the inequality
\begin{equation}
\label{ineq:bly}
\lambda_1(\Sigma^n,g)\leqslant 4n\frac{m+1}{m}\left(\int_{\Sigma^n}\phi^*(\omega_{FS})\wedge\omega_g^{n-1}\right)/\left(\int_{\Sigma^n}\omega_g^n\right),
\end{equation}
where $\omega_{FS}$ is the Fubini-Study form on $\mathbb CP^m$, and $\omega_g$ is the K\"ahler form of $g$. 
\end{theorem}
Above we assume that the Fubini-Study form $\omega_{FS}$ is normalised so that the diameter of $\mathbb CP^m$ equals $\pi/2$, see Section~\ref{prem} for the details on the notation used. The quotient of the integrals on the right hand-side of inequality~\eqref{ineq:bly}, that is
\begin{equation}
\label{def:deg}
d([\phi],[\omega_g]):=\left(\int_{\Sigma^n}\phi^*(\omega_{FS})\wedge\omega_g^{n-1}\right)/\left(\int_{\Sigma^n}\omega_g^n\right),
\end{equation}
is a homological invariant, called the {\em holomorphic degree}. It depends only on the cohomology class $[\omega_g]$ and the action of $\phi$ on $2$-cohomology $\phi^*:H^2(\mathbb CP^m,\mathbb Q)\to H^2(\Sigma^n,\mathbb Q)$, and is strictly positive when $\phi:\Sigma^n\to\mathbb CP^n$ is non-constant. Consider the set $\mathcal K_\Omega(\Sigma^n,J)$ of K\"ahler metrics on $(\Sigma^n,J)$ whose K\"ahler forms represent a given de Rham cohomology class $\Omega\in H^2(\Sigma^n,\mathbb Q)$. Then Theorem~\ref{t:bly} says that the first Laplace eigenvalue $\lambda_1(\Sigma^n,g)$ on a projective manifold $\Sigma^n$ is  bounded as the metric $g$ ranges in $\mathcal K_\Omega(\Sigma^n,J)$. Since metrics $g\in\mathcal K_\Omega(\Sigma^n,J)$ have the same volume, this statement actually gives a bound for the scale invariant quantity 
$$
\Lambda_1(g)=\lambda_1(\Sigma^n,g)\mathit{Vol}_g(\Sigma^n)^{1/n},
$$ 
where $n$ is the complex dimension. The restriction to a class of metrics $\mathcal K_\Omega(\Sigma^n, J)$ is necessary for such a bound to hold. Indeed, by the results of Colbois and Dodziuk~\cite{CD94}, see also Lohkamp~\cite{Loh}, the quantity $\Lambda_1(g)$ is unbounded when $n>1$ and $g$ ranges over all Riemannian metrics. Theorem~\ref{t:bly} also implies that the Fubini-Study metric on the projective space $\mathbb CP^m$ maximises the first Laplace eigenvalue in its K\"ahler class. This result has been generalised by Arezzo, Ghigi, and Loi~\cite{AGL} to the setting of K\"ahler manifolds that admit holomorphic stable vector bundles over $M$ with sufficiently many sections, under an appropriate stability condition. In particular, they show that the symmetric K\"ahler-Einstein metrics on the Grassmannian spaces also maximize the first Laplace eigenvalue in their K\"ahler classes. Moreover, as is shown in~\cite{BG13}, so do symmetric K\"ahler-Einstein metrics on Hermitian symmetric spaces of compact type. Related to this circle of questions extremal eigenvalue problems have been considered in~\cite{AJK}.

The purpose of this paper is to prove inequalities analogous to Theorem~\ref{t:bly} for higher Laplace eigenvalues, answering the question raised by Yau~\cite{Yau96}. We also obtain inequalities for higher Laplace eigenvalues on analytic subvarieties in K\"ahler manifolds.

\subsection{Bounds for higher Laplace eigenvalues}
For a Riemannian metric $g$ on a closed manifold $\Sigma^n$, we denote by
$$
0=\lambda_0(\Sigma^n,g)<\lambda_1(\Sigma^n,g)\leqslant\lambda_2(\Sigma^n,g)\leqslant\ldots\leqslant\lambda_k(\Sigma^n,g)\leqslant\ldots
$$
the eigenvalues of the Laplace-Beltrami operator $\Delta_g$, repeated with respect to multiplicity. Our main result is the following version of Theorem~\ref{t:bly} for all Laplace eigenvalues.
\begin{theorem}
\label{t1}
Let $(\Sigma^n,J)$ be an $n$-dimensional closed complex manifold, and let $\phi:\Sigma^n\to\mathbb CP^m$ be a non-constant holomorphic map. Then for any K\"ahler metric $g$ on $\Sigma^n$ its Laplace eigenvalues satisfy the following inequalities
\begin{equation}
\label{t1:ineq}
\lambda_k(\Sigma^n,g)\leqslant C(n,m)d([\phi],[\omega_g])k\qquad\text{for any}\quad k\geqslant 1,
\end{equation}
where $C(n,m)>0$ is a constant that depends on the dimensions $n$ and $m$ only, and $d([\phi],[\omega_g])$ is the holomorphic degree defined by relation~\eqref{def:deg}.
\end{theorem}
To our knowledge, Theorem~\ref{t1} is the first rigorous result in the literature that gives bounds for higher Laplace eigenvalues of K\"ahler metrics in a fixed K\"ahler class. Note that unlike in Theorem~\ref{t:bly},  we do not assume that a holomorphic map $\phi:\Sigma^n\to\mathbb CP^m$ is an immersion and do not impose any hypotheses on the image $\phi(\Sigma^n)$ in Theorem~\ref{t1}. In complex dimension one our theorem implies a celebrated result of Korevaar~\cite{Kor}: for any Hermitian metric $g$ on a complex curve $\Sigma^1$ the Laplace eigenvalues satisfy the inequalities
\begin{equation}
\label{ineq:kor}
\lambda_k(\Sigma^1,g)\mathit{Vol}_g(\Sigma^1)\leqslant C_*\deg(\phi)k\qquad\text{for any}\quad k\geqslant 1,
\end{equation}
where $\phi:\Sigma^1\to\mathbb CP^1$ is an arbitrary non-constant holomorphic map, and $C_*$ is a universal constant. Indeed, by the change of variables in integral formula we obtain
\begin{equation}
\label{d:deg:1}
d([\phi],[\omega_g])=\deg(\phi)\left(\mathit{Vol}(\mathbb CP^1)/\mathit{Vol}_g(\Sigma^1)\right)
\end{equation}
for an arbitrary non-constant holomorphic map $\phi:\Sigma^1\to\mathbb CP^1$. Now Korevaar's inequalities~\eqref{ineq:kor} follow directly from Theorem~\ref{t1}. As is known~\cite{GH}, for any complex curve $\Sigma^1$ there exists a non-constant holomorphic map $\phi:\Sigma^1\to\mathbb CP^1$ whose degree is not greater than $\gamma+1$, where $\gamma$ is the genus of $\Sigma^1$, and inequalities~\eqref{ineq:kor} imply the bounds
$$
\lambda_k(\Sigma^1,g)\mathit{Vol}_g(\Sigma^1)\leqslant C_*(\gamma+1)k\qquad\text{for any}\quad k\geqslant 1,
$$ 
for an arbitrary Riemannian metric on $\Sigma^1$. Theorem~\ref{t1} can be viewed as a natural generalisation to higher dimensional K\"ahler manifolds of Korevaar's result, and in particular, answers the question on the existence of bounds in K\"ahler classes for higher eigenvalues, raised by Yau in~\cite[p.~170]{Yau96}.

For fibrations $\phi:\Sigma^n\to\mathbb CP^1$ Theorem~\ref{t1} yields a version for higher Laplace eigenvalues of the inequality for the first Laplace eigenvalue by Li and Yau in 1982, see~\cite[Theorem~3]{LiYau82}. Related questions have been also discussed by Gromov in~\cite{Gro}. In this case the quantity $d([\phi],[\omega_g])$ takes a form similar to~\eqref{d:deg:1}: up to a constant it is the ratio $\deg(\phi)/\mathit{Vol}_g(\Sigma^n)$, where $\deg(\phi)$ is understood as the volume of the generic fiber of $\phi$. By considering fibrations over complex projective spaces, Theorem~\ref{t1} can be used to obtain bounds for all Laplace eigenvalues on not necessarily projective manifolds. For instance, all K\"ahler surfaces of algebraic dimension one are non-projective and elliptic, see~\cite{BPV}. In particular, they admit non-constant holomorphic maps to $\mathbb CP^1$, and hence, satisfy the hypotheses of Theorem~\ref{t1}. The examples include certain $K3$ surfaces and certain complex $2$-tori. Thus, in many instances we have a positive answer to the following outstanding question.
\begin{oquest}
Let $(\Sigma^n,J)$ be a closed K\"ahler manifold. Are Laplace eigenvalues bounded in every K\"ahler class $K_\Omega(\Sigma^n,J)$ of K\"ahler metrics?
\end{oquest}

Note that when the complex dimension $n>1$, the inequality in Theorem~\ref{t1} is not compatible with the Weyl asymptotic law 
$$
\lambda_k(\Sigma^n,g)\mathit{Vol}_g(\Sigma^n)^{1/n}\sim C(n)k^{1/n}\qquad\text{as~ }k\to +\infty,
$$
in the sense that the index $k$ occurs in it with the "wrong" power. However, inequality~\eqref{t1:ineq} can not be improved to the inequality where $k$ is replaced by $k^{1/n}$ on the right hand-side in~\eqref{t1:ineq}. For otherwise, passing to the limit, the Weyl law would imply the bound $\mathit{Vol}_g(\Sigma^n)^{-1/n}\leqslant C(n)d$, where $d$ is the holomorphic degree, which can not hold. To see the latter consider a fibration $\phi:\mathbb CP^1\times\Sigma_0^{n-1}\to\mathbb CP^1$ that forgets the second factor. Equipping it with the product metric $g_{FS}\oplus g_0$, we conclude that the holomorphic degree $d$ does not depend on a metric $g_0$ on $\Sigma_0^{n-1}$, and arrive at a contradiction with the hypothetical bound. The above discussion leads to the following question. 
\begin{oquest}
Are there bounds for Laplace eigenvalues in K\"ahler classes that are compatible with the asymptotic eigenvalue behaviour?
\end{oquest}

A few words about the proof of Theorem~\ref{t1}.  The main argument uses ingredients originating from the work of Korevaar~\cite{Kor}, and developed further by Grigoryan, Netrusov, and Yau in~\cite{GNY}. The novelty of our approach is an improved construction of test-functions that allows us to obtain eigenvalue bounds in K\"ahler classes of metrics in terms of the holomorphic degree only. The main idea is motivated by the construction used in~\cite{Ko17}. We describe it in Section~\ref{proof:t1}.

\subsection{Examples and further discussion}
Theorem~\ref{t1} applies to many homogeneous K\"ahler manifolds, and shows that all Laplace eigenvalues are bounded in K\"ahler classes on them. For example, complex Grassmannians are holomorphically and isometrically embedded into the projective spaces by the standard Pl\"ucker embedding. In more detail, for a finite-dimensional complex vector-space $W$ we denote by $\Gr (r,W)$ the Grassmannian of $r$-dimensional subspaces in $W$. The {\em Pl\"ucker embedding} $\Gr (r,W)\to \mathbb P(\wedge^rW)$ is defined by 
\begin{equation}
\label{def:plucker}
\Gr (r,W)\ni L\longmapsto [e_1\wedge\ldots\wedge e_r]\in\mathbb P(\wedge^r W),
\end{equation}
where $e_1,\ldots,e_r$ is a basis in an $r$-dimensional subspace $L\subset W$. Other examples include irreducible Hermitian symmetric spaces of compact type; by~\cite{Ca} they can be holomorphically and isometrically embedded into $(\mathbb CP^m, cg_{FS})$ for some integer $m>0$ and real number $c>0$. 

Holomorphic maps into projective spaces often occur via the so-called {\em Kodaira maps}. In more detail, let $E$ be a holomorphic vector bundle of rank $r$ over a closed K\"ahler manifold $\Sigma^n$. Let $V=H^0(E)$ be a space of global holomorphic sections of $E$, and let $N$ be its dimension. Suppose that $E$ is globally generated, and for $p\in\Sigma^n$ denote by $V_p\subset V$ the subspace formed by sections that vanish at $p$; its dimension equals $N-r$. The {\em Kodaira map} $\kappa_E:\Sigma^n\to\Gr (r,V^*)$ is defined by sending $p\mapsto\Ann (V_p)$, where $\Ann(V_p)$ is the annihilator subspace of $V_p$,
$$
\Ann(V_p)=\left\{\lambda\in V^*: \lambda\equiv 0\text{ on }V_p\right\}.
$$
Composing it with the Pl\"ucker embedding~\eqref{def:plucker}, we obtain an embedding $K_E:\Sigma^n\to\mathbb CP^{m}$, where $(m+1)$ is the binomial coefficient $\binom{N}{r}$. Besides, as is known, see~\cite{AGL,GH}, the pull-back of the Fubini-Study form $\omega_{FS}$ on $\mathbb CP^m$ represents a multiple of the first Chern class $c_1(E)$, that is $K^*_E([\omega_{FS}])=c\cdot c_1(E)$, where $c>0$ is a constant that depends on normalisation conventions only. Thus, we arrive at the following consequence of Theorem~\ref{t1}.
\begin{corollary}
\label{cor:agl}
Let $E$ be a holomorphic globally generated vector bundle over a closed K\"ahler manifold $(\Sigma^n,J)$. Then for any K\"ahler metric $g$ on $\Sigma^n$ its Laplace eigenvalues satisfy the following inequalities
$$
\lambda_k(\Sigma^n,g)\leqslant C(n,r,N)\left(\left(\int_{\Sigma^n}c_1(E)\cup[\omega_g]^{n-1}\right)/\left(\int_{\Sigma^n}[\omega_g]^n\right)\right)k
$$
for any $k\geqslant 1$, where the constant $C(n,r,N)$ depends on the dimension $n$ of $\Sigma^n$, rank $r$ of $E$, and $N=\dim H^0(E)$ only.
\end{corollary}
The hypotheses of Corollary~\ref{cor:agl} are close to the setting considered by Arezzo, Ghigi, and Loi in~\cite{AGL}, where the authors obtain bounds for the first Laplace eigenvalue. However, unlike the main result in~\cite{AGL}, we are not concerned with the value of the constant $C(n,r,N)$ and do not require any assumption on the stability of the Gieseker point of $E$ in Corollary~\ref{cor:agl}.

We end with a discussion of the version of Theorem~\ref{t1} for analytic subvarieties in K\"ahler manifolds. Let $M^{n+l}$ be a closed $(n+l)$-dimensional K\"ahler manifold, and let $\Sigma^n\subset M^{n+l}$ be an irreducible analytic subvariety whose regular locus $\Sigma_*^n$, that is the complement of the singular set, has complex dimension $n$. Any K\"ahler metric $g$ on $M^{n+l}$ induces an incomplete K\"ahler metric $g_\Sigma$ on $\Sigma^n_*$. We consider the Laplace operator defined on compactly supported $C^2$-smooth  functions on  $\Sigma^n_*$. In Section~\ref{variety} we explain that this operator is essentially self-adjoint and has discrete spectrum. The following statement gives bounds for Laplace eigenvalues of $\Sigma^n$ that are uniform over K\"ahler metrics in a fixed K\"ahler class on $M^{n+l}$.

\begin{theorem}
\label{t2}
Let $(M^{n+l},J)$ be an $(n+l)$-dimensional closed complex manifold, and let $\phi:M^{n+l}\to\mathbb CP^m$ be a holomorphic map. Let $\Sigma^n\subset M^{n+l}$ be an irreducible analytic subvariety such that the map $\phi$ is non-constant on the regular locus $\Sigma^n_*$. Then for any K\"ahler metric $g$ on $M^{n+l}$ the Laplace eigenvalues of $(\Sigma^n,g_{\Sigma})$ satisfy the following inequalities
$$
\lambda_k(\Sigma^n,g_\Sigma)\leqslant C(n,m)\left(\left(\int_{\Sigma^n}\phi^*(\omega_{FS})\wedge\omega_g^{n-1}\right)/\left(\int_{\Sigma^n}\omega_g^n\right)\right)k
$$
for any $k\geqslant 1$, where $C(n,m)$ is the constant that depends on $n$ and $m$ only, and $\omega_g$ is the K\"ahler form of $g$ on $M^{n+l}$.
\end{theorem}

In particular, when the Hodge number $h^{1,1}(M^{n+l})$ equals one, Theorem~\ref{t2} gives eigenvalue bounds that are also uniform over both subvarietes $\Sigma^n$ of $M^{n+l}$ and all K\"ahler metrics $g$ on $M^{n+l}$ of unit volume. To our knowledge this statement is new even for algebraic varieties.


\section{Preliminaries and notation}
\label{prem}

\subsection{Geometry of the complex projective space}
Let $\mathbb CP^m$ be a complex projective space equipped with the Fubini-Study metric $g_{FS}$. We assume that the Fubini-Study metric is normalised such that the diameter of $\mathbb CP^m$ equals $\pi/2$. Viewing $\mathbb CP^m$ as the collection of $1$-dimensional subspaces in $\mathbb C^{m+1}$, this convention means that the pull-back $\pi^*\omega_{FS}$ of the corresponding K\"ahler form $\omega_{FS}$ satisfies the relation
$$
\pi^*\omega_{FS}=\frac{i}{2}\partial\bar\partial\log\abs{Z}^2,\qquad\text{where}\quad\abs{Z}^2=\sum\limits_{\ell=0}^m\abs{z_\ell}^2,
$$
and $\pi:\mathbb C^{m+1}\backslash\{0\}\to\mathbb CP^m$ is a natural projection. Recall that the distance function $\dist_{FS}$ corresponding to the Fubini-Study metric satisfies the following relation
\begin{equation}
\label{rel:key1}
\cos (\dist_{FS}([Z],[W]))=\frac{\abs{\langle Z,W\rangle}}{\abs{Z}\abs{W}},
\end{equation}
where $Z$, $W\in\mathbb C^{m+1}\backslash\{0\}$, and the brackets $\langle\cdot,\cdot\rangle$ denote the standard Hermitian product on $\mathbb C^{m+1}$. In the sequel we denote by $\Cut_{[Z]}$ the cut locus of a point $[Z]\in\mathbb CP^m$. The analysis of geodesics in $\mathbb CP^m$ shows that it is a hyperplane formed by all $1$-dimensional subspaces $[W]$ orthogonal to $[Z]$. We refer to~\cite{AG,BGM}, where these and related facts are discussed in detail.

As is well-known~\cite{GH}, any biholomorphic map $\mathbb CP^m\to\mathbb CP^m$ has the form
\begin{equation}
\label{f:bihol}
[Z]\longmapsto [CZ],\qquad\text{where}\quad Z\in\mathbb C^{m+1}\backslash\{0\}, \text{ ~and~ } C\in\Gl_{m+1}(\mathbb C).
\end{equation}
In particular, we see that the group of biholomorphisms of $\mathbb CP^m$ is connected, and hence, any biholomorphism induces the identity map on cohomology.

Recall that the isometry group of $\mathbb CP^m$ with respect to the Fubini-Study metric is formed by biholomorphisms~\eqref{f:bihol} such that $C\in\SU_{m+1}$. Consider the moment map $\tau:\mathbb CP^m\to\mathfrak{su}_{m+1}^*$ for the action of the isometry group; it satisfies the relation
$$
d(\tau,X)=-\imath_{\xi_X}\omega_{FS},
$$
where $X\in\mathfrak{su}_{m+1}$ and $\xi_X$ is the fundamental vector field for the action on $\mathbb CP^m$. Identifying the dual space $\mathfrak{su}_{m+1}^*$ with the Lie algebra $\mathfrak{su}_{m+1}$ by means of the Killing scalar product 
$$
(X,Y)=\mathit{trace}(X^*Y)=-\mathit{trace}(XY),
$$
we may assume that $\tau$ takes values in $\mathfrak{su}_{m+1}$. Then, in standard homogeneous coordinates $[Z]=[z_0:z_1:\ldots:z_m]$ on $\mathbb CP^m$, it can be written as
\begin{equation}
\label{mm:f}
\tau([Z])={i}\frac{ZZ^*}{Z^*Z}={i}\left(\frac{z_i\bar z_j}{\abs{Z}^2}\right)_{0\leqslant i,j \leqslant m}\!\!\!\!\!\!\!\!\!\!\!\!\!\!\!\!\!\!\!\!\!\!\!\!\!\!\!,
\end{equation}
see details in~\cite{AGL,MDS98}. Note also that the moment map satisfies the following identity:
\begin{equation}
\label{rel:mommap}
\omega_{FS}=-\frac{i}{2}\sum\limits_{j,\ell=0}^{m}d\tau_{j\ell}\wedge d\tau_{\ell j},
\end{equation}
where $\tau_{j\ell}$ are entries of the matrix $\tau$. Due to the equivariance properties, it is sufficient to see that this relation holds at one point, for example at $[1:0:\ldots:0]$, where it can be verified in a straightforward manner.

The following statement is implicitly contained in~\cite{BLY,AGL}. We include a proof for the completeness of the exposition.
\begin{lemma}
\label{denom:top}
Let $(\Sigma^n,g,J)$ be a K\"ahler manifold, and let $\phi:\Sigma^n\to\mathbb CP^m$ be a holomorphic map. Then the gradient of the matrix-valued map $\tau\circ\phi:\mathbb CP^m\to\mathbb C^{(m+1)\times (m+1)}$ satisfies the relation
$$
\abs{\nabla(\tau\circ\phi)}^2\omega_g^n=n\phi^*(\omega_{FS})\wedge\omega_g^{n-1},
$$
where $\tau:\mathbb CP^m\to\mathfrak{su}_{m+1}$ is the moment map, and $\omega_g$ is the K\"ahler form of a metric $g$.
\end{lemma}
\begin{proof}
Recall that for any real $(1,1)$-form $\alpha$ on $\Sigma^n$ the following relation holds
\begin{equation}
\label{in:rel1}
(\alpha,\omega_g)\omega_g^n=n\alpha\wedge\omega_g^{n-1},
\end{equation}
where $\omega_g$ is the K\"ahler form on $\Sigma^n$, and the brackets $(\cdot,\cdot)$ denote the induced Euclidean product on $(1,1)$-forms. Choosing
$$
\alpha=i(\partial\varphi\wedge\bar\partial \bar\varphi+\partial\bar\varphi\wedge\bar\partial\varphi)
$$
for a smooth $\mathbb C$-valued function $\varphi$ on $\Sigma^n$, from~\eqref{in:rel1} we obtain
\begin{equation}
\label{in:rel2}
\abs{\nabla\varphi}^2\omega_g^n=in(\partial\varphi\wedge\bar\partial \bar\varphi+\partial\bar\varphi\wedge\bar\partial\varphi)\wedge\omega_g^{n-1}.
\end{equation}
Now let $\varphi_{j\ell}$ be the entries of the matrix $\tau\circ\phi$. Then, using~\eqref{in:rel2}, we obtain
\begin{multline*}
\abs{\nabla\varphi_{j\ell}}^2\omega_g^n=in(\partial\varphi_{j\ell}\wedge\bar\partial \bar\varphi_{j\ell}+\partial\bar\varphi_{j\ell}\wedge\bar\partial\varphi_{j\ell})\wedge\omega_g^{n-1}=\\
in\phi^*(\partial\tau_{j\ell}\wedge\bar\partial \bar\tau_{j\ell}+\partial\bar\tau_{j\ell}\wedge\bar\partial\tau_{j\ell})\wedge\omega_g^{n-1}=-in\phi^*(\partial\tau_{j\ell}\wedge\bar\partial \tau_{\ell j}+\partial\tau_{\ell j}\wedge\bar\partial\tau_{j\ell})\wedge\omega_g^{n-1},
\end{multline*}
where in the second relation we used the hypothesis that $\phi$ is holomorphic, and in the last -- that the moment map $\tau$ takes values in $\mathfrak{su}_{m+1}$. Combining the above with relation~\eqref{rel:mommap}, we conclude that
$$
\abs{\nabla(\tau\circ\phi)}^2\omega_g^n=\sum\limits_{i,j=0}^m\abs{\nabla\varphi_{j\ell}}^2\omega_g^n=n\phi^*(\omega_{FS})\wedge\omega_g^{n-1}.
$$
Thus, the statement is demonstrated.
\end{proof}

\subsection{First eigenfunctions and holomorphic vector fields.}
A well-known result by Matsushima~\cite{Ma57} establishes a relationship between first Laplace eigenfunctions and holomorphic vector fields on K\"ahler-Einstein manifolds of positive scalar curvature. In more detail, for any first eigenfunction $f$ its gradient $\grad f$ is a holomorphic non-Killing vector field, see the discussion in~\cite{AJK}. In the sequel we will need an explicit description of the gradient flow of a particular eigenfunction on the complex projective space $\mathbb CP^m$.

Given  a $1$-dimensional subspace $[W]$ in $\mathbb C^{m+1}$ and a real number $t>0$ we consider a $\mathbb C$-linear operator $\Theta_{t,[W]}:\mathbb C^{m+1}\to\mathbb C^{m+1}$ defined by the following relation
$$
\Theta_{t,[W]}Z=\left\{
\begin{array}{ccc}
Z, & \text{ if } & Z\in [W],\\
tZ, & \text{ if } & \langle Z,W\rangle=0, 
\end{array}
\right.
$$
where the brackets $\langle\cdot,\cdot\rangle$ denote the standard Hermitian product on $\mathbb C^{m+1}$. By $\theta_{t,[W]}$ we denote the induced biholomorphism $\mathbb CP^m\to\mathbb CP^m$, given by
$$
\theta_{t,[W]}[Z]=[\Theta_{t,[W]}Z].
$$
It is clear that a point $[W]\in\mathbb CP^m$ as well as the points $[Z]$ corresponding to $1$-dimensional subspaces orthogonal to $[W]$ are fixed points of $\theta_{t,[W]}$ for any $t>0$. 

Now for a given $1$-dimensional subspace $[W]\in\mathbb CP^m$ we consider a function $\varphi_{[W]}:\mathbb CP^m\to\mathbb R$ defined as
\begin{equation}
\label{def:model}
\varphi_{[W]}([Z])=\frac{\abs{\langle Z,W\rangle}^2}{\abs{Z}^2\abs{W}^2},\qquad\text{where~ }[Z]\in\mathbb CP^m,
\end{equation}
and the brackets $\langle\cdot,\cdot\rangle$ denote the standard Hermitian product on $\mathbb C^{m+1}$. The maximum of $\varphi_{[W]}$ is achieved at the point $[W]$, and the minimum -- at the cut locus $\Cut_{[W]}$.  As is well-known~\cite{BGM,BLY}, the function $\varphi_{[W]}-(1/(m+1))$ is a first eigenfunction of the Laplace-Beltrami operator on $\mathbb CP^m$.

The following lemma describes a relationship between the function $\varphi_{[W]}$ and the family of biholomorphisms $\theta_{t,[W]}$. Its proof is a straightforward exercise, but we include the details for reader's convenience.
\begin{lemma}
\label{t1:l1}
For any $[W]\in\mathbb CP^m$ the family of biholomorphisms $\theta_{e^{-2\tau},[W]}$, where $\tau\in\mathbb R$, is the gradient flow of the function $\varphi_{[W]}$. 
\end{lemma}
\begin{proof}
By the equivariance properties,
$$
\varphi_{[CW]}([CZ])=\varphi_{[W]}([Z])\quad\text{ and }\quad\theta_{t,[CW]}([CZ])=[C]\theta_{t,[W]}([Z]),
$$
where $[Z], [W]\in\mathbb CP^m$ and $C\in\SU_{m+1}$, it is sufficient to prove the statement of the lemma for one point $[W]\in\mathbb CP^m$, for example, when $[W]=[1:0:\ldots:0]$. For the rest of the proof we assume that $[W]$ is chosen in this way, and denote by $f$ the function $\varphi_{[W]}$. Since the point $[W]$ and its cut locus $\Cut_{[W]}$ are critical sets of $f$, and the vector field
\begin{equation}
\label{l1:def:vf}
X_{[Z]}=\left.\frac{d}{d\tau}\right|_{\tau=0}\!\!\!\!\!\theta_{e^{-2\tau},[W]}([Z])
\end{equation}
vanishes on these sets, it remains to verify the hypothesis that $X$ is the gradient of $f$ 
on the complement of $[W]\cup\Cut_{[W]}$ only. We denote the complement of $\Cut_{[W]}$ by $U_0$; it is formed by points $[Z]=[z_0:z_1:\ldots:z_m]$ such that $z_0\ne 0$. In the coordinate chart
$$
U_0\ni [z_0:z_1:\ldots:z_m]\longmapsto \left(\frac{z_1}{z_0},\ldots,\frac{z_m}{z_0}\right)\in\mathbb C^m
$$
the function $f=\varphi_{[W]}$ takes the form
\begin{equation}
\label{l1:func}
f(\zeta_1,\ldots,\zeta_m)=\left(1+\sum_{i=1}^m\abs{\zeta_i}^2\right)^{-1},
\end{equation}
and the biholomorphism $\theta_{t,[W]}$ acts as a dilation, $\zeta\mapsto t\zeta$. Thus, by relation~\eqref{l1:def:vf}, the vector field $X$ takes the form $X(\zeta)=-2\zeta$.
The hypothesis that $X$ is the gradient of $f$ is equivalent to the relation $\imath_X\omega_{FS}=Jdf$, where we assume that the complex structure $J$ acts on the $1$-form $df$ as $-df(J\cdot)$. The latter can be also re-written in the following form
\begin{equation}
\label{l1:ve:grad}
\imath_{X^{1,0}}\omega_{FS}={i}(df)^{0,1},
\end{equation}
where $X^{1,0}$ and $(df)^{0,1}$ stand for $(1,0)$- and $(0,1)$-parts of $X$ and $df$ respectively. Using the formula for the Fubini-Study metric in these coordinates,
$$
\omega_{FS}=\frac{i}{2}\partial\bar\partial\log\left(1+\sum_{i=1}^m\abs{\zeta_i}^2\right),
$$
and formula~\eqref{l1:func} for the function $f$, relation~\eqref{l1:ve:grad} can be verified in a straightforward fashion.
\end{proof}

By relation~\eqref{rel:key1} the metric balls $B_{[W]}(r)$ in $\mathbb CP^m$ with respect to the Fubini-Study metric are precisely the sets
\begin{equation}
\label{rel:mball}
B_{[W]}(r)=\{[Z]\in\mathbb CP^m: \varphi_{[W]}([Z])<\cos^2r\},
\end{equation}
where $r\in [0,\pi/2]$.
The next lemma is essentially a consequence of Lemma~\ref{t1:l1}, but we state it separately for the convenience of references.
\begin{lemma}
\label{t1:l2}
For any point $[W]\in\mathbb CP^m$ and any $t>0$ the biholomorphism $\theta_{t,[W]}:\mathbb CP^m\to\mathbb CP^m$ maps a metric ball $B_{[W]}(r)$ in the Fubini-Study metric, where $r\in (0,\pi/2)$, onto the metric ball $B_{[W]}(\rho)$ such that the radii are related as $t\tan r=\tan\rho$. 
\end{lemma}
\begin{proof}
Since $\theta_{e^{-2\tau},[W]}$ is the gradient flow of $\varphi_{[W]}$, it clearly maps the level sets of $\varphi_{[W]}$ into themselves preserving the natural order, given by the values of the function $\varphi_{[W]}$. The combination of these facts together with relation~\eqref{rel:mball} shows that a metric ball $B_{[W]}(r)$ in the Fubini-Study metric is mapped by $\theta_{t,[W]}$ onto a concentric metric ball $B_{[W]}(\rho)$. The relationship between the radii can be, for example, derived from the local representation~\eqref{l1:func} for the function $\varphi_{[W]}$. In more detail, assuming that $[W]=[1:0:\ldots:0]$ and using the notation in the proof of Lemma~\ref{t1:l1}, we see that if $\zeta\in\partial B_{[W]}(r)$, then the combination of~\eqref{rel:key1} and~\eqref{l1:func} yields the following relation:
\begin{equation}
\label{l2:eq}
\frac{1}{1+\abs{\zeta}^2}=\cos^2r\quad\Leftrightarrow\quad \abs{\zeta}^2=\tan^2r,
\end{equation}
where $\zeta\in\mathbb C^m$ and $\abs{\zeta}^2=\sum\abs{\zeta_i}^2$. Since in these coordinates the biholomorphism $\theta_{t,[W]}$ is the dilation $\zeta\mapsto t\zeta$, the radius $\rho$ of the image ball satisfies the relation
$$
\frac{1}{1+t^2\abs{\zeta}^2}=\cos^2\rho\quad\Leftrightarrow\quad t^2\abs{\zeta}^2=\tan^2\rho.
$$
Comparing the last relation with~\eqref{l2:eq}, we obtain $t^2\tan^2r=\tan^2\rho$.
\end{proof}


\section{Proof of Theorem~\ref{t1}}
\label{proof:t1}
\subsection{Geometry of  metric measure spaces}
The proofs of Theorems~\ref{t1} and~\ref{t2} are based on a statement that guarantees the existence of an arbitrary number of disjoint sets carrying a sufficient amount of volume. Below we give a brief account on it in the setting of metric measure spaces. By $(X, d)$ we denote a separable metric space; the ball $B_p(r)$ is a subset of the form $\{x\in X: d(x,p)<r\}$. We start with the following definition.
\begin{defin}
For an integer $N> 1$ a metric space $(X,d)$ is said to satisfy the {\em $N$-covering property}, if each ball $B_p(r)$ can be covered by $N$ balls of radius $r/2$.
\end{defin}

As the example below shows, complete Riemannian manifolds of non-negative Ricci curvature can be viewed as metric spaces that satisfy the $N$-covering property such that the value $N$ depends on the dimension of a manifold only.
\begin{example}[Spaces of non-negative Ricci curvature]
\label{nnrc}
Let $(M^\ell,h)$ be a complete Riemannian manifold of non-negative Ricci curvature, and $\dist(\cdot,\cdot)$ be its distance function. Recall that the value $\dist(p,q)$ is defined as the infimum of lengths of all smooth paths joining $p$ and $q$. We claim that the {\em metric space $(M^\ell,\dist)$ satisfies the $N$-covering property with $N=9^\ell$}. Indeed, recall that by the relative volume comparison theorem, see~\cite{Cha}, the volumes of concentric balls with respect to $\dist$, satisfy the relation
\begin{equation}
\label{rvc}
\frac{\mathit{Vol}(B_p(R))}{\mathit{Vol}(B_p(r))}\leqslant\left(\frac{R}{r}\right)^\ell,\qquad\text{where~ } 0<r\leqslant R.
\end{equation}
Now for a given ball $B_p(r)$ let $\mathcal C=\{B_{p_i}(r/4)\}$ be a maximal family of disjoint balls centred at $p_i\in B_p(r)$. Then the family $\{B_{p_i}(r/2)\}$ is a covering of $B_p(r)$. Thus, for a proof of our claim it is sufficient to estimate the cardinality of the covering $\mathcal C$, $m=\card\mathcal C$. Let $p_{i_0}$ be a point such that
$$
\mathit{Vol}(B_{p_{i_0}}(r/4))=\min\left\{\mathit{Vol}(B_{p_{i}}(r/4)): i=1,\ldots, m\right\}.
$$
Then we obtain
\begin{multline}
\label{aux:est}
m\mathit{Vol}(B_{p_{i_0}}(r/4))\leqslant \sum_{i=1}^m\mathit{Vol}(B_{p_{i}}(r/4))=\mathit{Vol}\left(\bigcup_iB_{p_{i}}(r/4)\right)\leqslant \\
\mathit{Vol}(B_{p}(5r/4))\leqslant \mathit{Vol}(B_{p_{i_0}}(9r/4)),
\end{multline}
where in the second inequality we used the fact that the centres $p_i$ of the balls belong to $B_p(r)$, and hence, the balls $B_{p_i}(r/4)$ lie in the ball $B_p(5r/4)$. In the last inequality in~\eqref{aux:est} we used the inclusion $B_p(5r/4)\subset B_{p_{i_0}}(9r/4)$. Now by inequality~\eqref{rvc}, we conclude
$$
m\leqslant\frac{\mathit{Vol}(B_{p_{i_0}}(9r/4))}{\mathit{Vol}(B_{p_{i_0}}(r/4))}\leqslant\left(\frac{9r/4}{r/4}\right)^\ell=9^\ell,
$$
and thus, the claim is demonstrated.
\end{example}

Developing the ideas of Korevaar~\cite{Kor}, Grigoryan, Netrusov, and Yau~\cite{GNY} showed that on certain metric spaces with such covering properties for any non-atomic finite measure one can always find a collection of disjoint sets carrying a controlled amount of measure. The geometry of such sets is also important. In general, they can not be chosen as metric balls, but can be chosen as the so-called annuli. By an annulus $A$ in $(X,d)$ we mean a subset of the following form
$$
\{x\in X: r\leqslant d(x,p)<R\},
$$
where $p\in X$ and $0\leqslant r<R<+\infty$. The real numbers $r$ and $R$ above are called the {\em inner} and {\em outer} radii respectively; the point $p$ is the centre of an annulus $A$. In addition, we denote by $2A$ the annulus
$$
\{x\in X:r/2\leqslant d(x,p)<2R\}.
$$
The following statement is the reformulation of~\cite[Corollary~3.2]{GNY}. It builds on the original results of Korevaar~\cite{Kor}, and this improvement is essential for our proof of Theorem~\ref{t1}.
\begin{prop}
\label{ds1}
Let $(X,d)$ be a separable metric space such that all balls $B_p(r)$ are precompact. Suppose that it satisfies the $N$-covering property for some $N>1$. Then for any finite non-atomic measure $\mu$ on $(X,d)$ and any positive integer $k$ there exists a collection of $k$ disjoint annuli $\{2A_i\}$ such that
\begin{equation}
\label{ds1:ineq}
\mu(A_i)\geqslant c\mu(X)/k\qquad\text{for any }1\leqslant i\leqslant k,
\end{equation}
where $c$ is a positive constant that depends on $N$ only. 
\end{prop}
The hypotheses in the proposition above are rather delicate and the conclusion can not be easily improved. For example, the power of the integer $k$ in inequality~\eqref{ds1:ineq} is optimal. The value of the constant $c$ in Proposition~\ref{ds1} can be chosen explicitly. In more detail, we can suppose that it is given by the relation $c^{-1}=8N^{12}$. This observation follows by examining the main argument in~\cite[Section 3]{GNY}, see the proof of~\cite[Lemma~3.4]{GNY}.

Proposition~\ref{ds1} and its ramifications have been used to obtain eigenvalue upper bounds for a number of eigenvalue problems under various hypotheses, see~\cite{GNY,Ko14,AH,Ko13,HaKo} and references therein. However, all these results are concerned with conformal eigenvalue bounds. The method that we use for a proof of Theorem~\ref{t1} is similar in the spirit to the argument in~\cite{Ko17}, and relies on a new construction of test-functions intimately linked to the geometry of $\mathbb CP^m$. We describe this construction below.

\subsection{Construction of test-functions}
We start with constructing auxiliary Lipschitz functions supported in metric balls and their complements. Our functions are modelled on the function $\varphi_{[W]}$,
$$
\varphi_{[W]}([Z])=\frac{\abs{\langle Z,W\rangle}^2}{\abs{Z}^2\abs{W}^2},\qquad\text{where~ }[Z]\in\mathbb CP^m,
$$
 and the construction uses the properties of the family of biholomorphisms $\theta_{t,[W]}$ described in Section~\ref{prem}. Recall that by relation~\eqref{rel:key1}, we have
\begin{equation}
\label{rel:key1n}
\varphi_{[W]}([Z])=\cos^2 (\dist_{FS}([Z],[W]))
\end{equation}
for any $[Z]$, $[W]\in\mathbb CP^m$. Thus, the restriction of the function $\varphi_{[W]}-(1/2)$ to the ball $B_{[W]}(\pi/4)$ gives a positive smooth function that vanishes on the boundary of the ball. For a given $R\in (0,\pi/4)$ we choose the value $t=t(R)>0$ such that $\theta_{t,[W]}$ maps the ball $B_{[W]}(2R)$ onto the ball $B_{[W]}(\pi/4)$; by Lemma~\ref{t1:l2} such a value $t$ exists and is unique. We define a function $\psi_{R,[W]}$ on the projective space $\mathbb CP^m$ by setting
\begin{equation}
\label{t1:f1}
\psi_{R,[W]}([Z])=\left\{
\begin{array}{lcc}
\varphi_{[W]}(\theta_{t,[W]}([Z]))-\displaystyle{\frac{1}{2}}, & \text{ if } & [Z]\in B_{[W]}(2R),\\
0, & \text{ if } & [Z]\notin B_{[W]}(2R).
\end{array}
\right.
\end{equation}
Clearly, it is a non-negative Lipschitz function, which is supported in the metric ball $B_{[W]}(2R)$ and is not greater than $(1/2)$ everywhere. The following auxiliary lemma says that it is bounded below away from zero on the smaller ball $B_{[W]}(R)$.
\begin{lemma}
\label{t1:l3}
For any $R\in(0,\pi/4)$ and any point $[W]\in\mathbb CP^m$ the function $\psi_{R,[W]}$ defined by relation~\eqref{t1:f1} satisfies the inequality
\begin{equation}
\label{t1:f1:key}
\psi_{R,[W]}([Z])\geqslant\frac{3}{10}\qquad\text{for any}\quad [Z]\in B_{[W]}(R).
\end{equation}
\end{lemma}
\begin{proof}
We follow the notation in the proof of Lemma~\ref{t1:l1}. First, by the equivariance properties it is sufficient to prove the lemma for the case when $[W]=[1:0:\ldots:0]$. Second, note that by Lemma~\ref{t1:l1} together with relation~\eqref{rel:key1n} the restriction of the function $\psi_{R,[W]}$ to the ball $B_{[W]}(R)$ achieves minimum on the boundary of the ball. Thus, for a proof of the lemma it is sufficient to show that inequality~\eqref{t1:f1:key} holds for any point $[Z]$ that belongs to the boundary of the ball $B_{[W]}(R)$. Let $B_{[W]}(\rho)$ be the image of $B_{[W]}(R)$ under the biholomorphism $\theta_{t,[W]}$; by Lemma~\ref{t1:l2} we have $\tan\rho=t\tan R$.

Following the argument in the proof of Lemma~\ref{t1:l1}, we reduce the considerations to the coordinate chart $U_0$, formed by the points $[Z]=[z_0:z_1:\ldots:z_m]$ such that $z_0\ne 0$. In this chart the ball $B_{[W]}(\rho)$ is represented by the Euclidean ball centred at the origin of radius $\tan\rho$, see relation~\eqref{l2:eq}. Thus, writing down the function $\varphi_{[W]}$ in these coordinates, for any $[Z]\in\partial B_{[W]}(R)$ we obtain
\begin{equation}
\label{l3:eq}
\varphi_{[W]}(\theta_{t,[W]}([Z]))=\frac{1}{1+\tan^2\rho}=\frac{1}{1+t^2\tan^2R}.
\end{equation}
Recall that the value $t=t(R)>0$ above is chosen such that the bihilomorphism $\theta_{t,[W]}$ maps the metric ball $B_{[W]}(2R)$ onto the ball $B_{[W]}(\pi/4)$; by Lemma~\ref{t1:l2} it equals $\tan^{-1}(2R)$. Thus, relation~\eqref{l3:eq} takes the form 
$$
\varphi_{[W]}(\theta_{t,[W]}([Z]))=\left(1+\frac{\tan^2R}{\tan^2(2R)}\right)^{-1}
$$ 
for any $[Z]\in\partial B_{[W]}(R)$ and any $R\in(0,\pi/4)$. It is straightforward to see that the right hand-side in the relation above, as a function of $R$, achieves its minimum $(4/5)$ when $R=0$. Thus, we conclude that
$$
\psi_{R,[W]}([Z])\geqslant\frac{4}{5}-\frac{1}{2}=\frac{3}{10}
$$
for any $[Z]\in B_{[W]}(R)$.
\end{proof}

Now we define a second auxiliary function supported in the complement of a given metric ball $B_{[W]}(r/2)$ in $\mathbb CP^m$. For a given $r\in (0,\pi/2)$ we choose the value $t=t(r)>0$ such that the biholomorphism $\theta_{t,[W]}$ maps the metric ball $B_{[W]}(r/2)$ in the Fubini-Study metric onto the ball $B_{[W]}(\pi/4)$. We define a function $\bar\psi_{r,[W]}$ on the projective space $\mathbb CP^m$ by setting
\begin{equation}
\label{t1:f2}
\bar\psi_{r,[W]}([Z])=\left\{
\begin{array}{lcc}
0, & \text{ if } & [Z]\in B_{[W]}(r/2),\\
(\varphi_{[W]}(\theta_{t,[W]}([Z]))+1)^{-1}-\displaystyle{\frac{2}{3}}, & \text{ if } & [Z]\notin B_{[W]}(r/2).
\end{array}
\right.
\end{equation}
Clearly, it is a non-negative Lipschitz function, which is supported in the complement of the metric ball $B_{[W]}(r/2)$ and is not greater than $(1/3)$ everywhere. The following statement is a version of Lemma~\ref{t1:l3} for the function $\bar\psi_{r,[W]}$.

\begin{lemma}
\label{t1:l4}
For any $r\in(0,\pi/2)$ and any point $[W]\in\mathbb CP^m$ the function $\bar\psi_{r,[W]}$ defined by relation~\eqref{t1:f2} satisfies the inequality
\begin{equation}
\label{t1:f2:key}
\bar\psi_{r,[W]}([Z])\geqslant\frac{1}{6}\qquad\text{for any}\quad [Z]\notin B_{[W]}(r).
\end{equation}
\end{lemma}
\begin{proof}
Let $t=t(r)>0$ be a real number such that the biholomorphism $\theta_{t,[W]}$ maps the ball $B_{[W]}(r/2)$ onto the ball $B_{[W]}(\pi/4)$; by Lemma~\ref{t1:l2} it equals $\tan^{-1}(r/2)$. By Lemma~\ref{t1:l1} the restriction of the function $\varphi_{[W]}(\theta_{t,[W]}([Z])$ to the complement of the ball $B_{[W]}(r)$ achieves its maximum on the boundary. Hence, the restriction of the function $\bar\psi_{r,[W]}$ to the complement of the ball $B_{[W]}(r)$ achieves its minimum on the boundary $\partial B_{[W]}(r)$, and for a proof of the lemma it is sufficient to show that inequality~\eqref{t1:f2:key} holds for any $[Z]\in\partial B_{[W]}(r)$. Following the line of argument in the proof of Lemma~\ref{t1:l3}, we obtain
$$
\varphi_{[W]}(\theta_{t,[W]}([Z]))=\left(1+\frac{\tan^2r}{\tan^2(r/2)}\right)^{-1}
$$ 
for any $[Z]\in\partial B_{[W]}(r)$ and any $r\in(0,\pi/2)$. An elementary computation shows that the right hand-side as a function of $r$ achieves its maximum $(1/5)$ when $r=0$. Thus, we conclude that
$$
\bar\psi_{r,[W]}([Z])\geqslant\frac{5}{6}-\frac{2}{3}=\frac{1}{6}
$$
for any $[Z]\notin B_{[W]}(r)$.
\end{proof}

Now consider annuli $A$ and $2A$ in $\mathbb CP^m$, which are complements of concentric metric balls, 
$$
A=B_{[W]}(R)\backslash B_{[W]}(r)\quad\text{and}\quad   2A=B_{[W]}(2R)\backslash B_{[W]}(r/2),
$$ 
where $0\leqslant r<R<\pi/4$ and $[W]\in\mathbb CP^m$. We define a function $u_A$ on $\mathbb CP^m$ by setting it to be the product $\psi_{R,[W]}\bar\psi_{r,[W]}$. Clearly, it is a Lipschitz function such that
$$
0\leqslant u_A\leqslant\frac{1}{6}\qquad\text{everywhere on~ }\mathbb CP^m.
$$
Besides, it is supported in the annulus $2A$, and is bounded away from zero on $A$, 
$$
u_A([Z])\geqslant\frac{1}{20}\qquad\text{ for any~ }[Z]\in A.
$$
We use the pull-backs of such functions as test-functions for the Rayleigh quotient to complete the proof of Theorem~\ref{t1} below.

\subsection{The estimate for the Laplace eigenvalues}
Now we prove Theorem~\ref{t1}. Recall that for any admissible test-function $u$ on $\Sigma^n$ the Rayleigh quotient $\mathcal R(u)$ is defined by the relation
$$
\mathcal R(u)=\left(\int_{\Sigma^n}\abs{\nabla u}^2d\mathit{Vol}_g\right)/\left(\int_{\Sigma^n}u^2d\mathit{Vol}_g\right).
$$
By the variational principle, see~\cite{Cha}, for a proof of the theorem it is sufficient for any $k\geqslant 1$  to construct a collection of $W^{1,2}$-orthogonal  $k+1$ Lipschitz test-functions $u_i$ such that
\begin{equation}
\label{t1:punch}
\mathcal R(u_i)\leqslant C(n,m)d([\phi],[\omega_g])k,
\end{equation}
where $d([\phi],[\omega_g])$ is the holomorphic  degree defined by the relation
$$
d([\phi],[\omega_g])=\left(\int_{\Sigma^n}\phi^*(\omega_{FS})\wedge\omega_g^{n-1}\right)/\left(\int_{\Sigma^n}\omega_g^n\right).
$$
We regard the complex projective space $\mathbb CP^m$ as a metric space with the distance function $\dist_{FS}$ induced by the Fubini-Study metric. By Example~\ref{nnrc} we conclude that the metric space $(\mathbb CP^m,\dist_{FS})$ satisfies the $N$-covering property with $N=9^{2m}$. We endow $(\mathbb CP^m,\dist_{FS})$ with a measure $\mu$ obtained as the push-forward of the volume measure $\mathit{Vol}_g$ on $\Sigma^n$ under a given holomorphic map $\phi:\Sigma^n\to\mathbb CP^m$. As is known~\cite{GH}, the pre-image $\phi^{-1}([W])$ of any point $[W]\in\mathbb CP^m$ is an analytic subvariety of $\Sigma^n$, which could be empty. Moreover, since $\phi$ is non-constant, its codimension is positive, and hence, the push-forward measure $\mu$ is non-atomic. Thus, Proposition~\ref{ds1} applies and we can find a collection $\{A_i\}$ of $k+1$ annuli on $\mathbb CP^m$ such that
\begin{equation}
\label{denom1}
\mu(A_i)\geqslant c\mu(\mathbb CP^m)/(k+1)\geqslant \mu(\mathbb CP^m)/2k\qquad\text{for any }i=1,\ldots,k,
\end{equation}
where the constant $c$ depends only on $m$, and the annuli $\{2A_i\}$ are disjoint.

We denote by $u_i$ the Lipschitz test-functions $u_{A_i}\circ\phi$, where $u_{A_i}$ are constructed above. In more detail, let $[W_i]$, $r_i$, and $R_i$ be the centre, the inner radius and the outer radius of $A_i$ respectively. Denote by $\psi_i$ the functions $\psi_{R_i,[W_i]}$, and by $\bar\psi_i$ the function $\bar\psi_{r_i,[W_i]}$, see the construction above. Then the function 
$$
u_i=\left\{
\begin{array}{cc}
(\psi_i\bar\psi_i)\circ\phi, &\text{ if~~ }r_i>0,\\
\psi_i\circ\phi, & \text{ if~~ }r_i=0,
\end{array}
\right.
$$ 
can be used as a test-function for the Rayleight quotient on $\Sigma^n$. Since the $u_i$'s are supported in the disjoint sets $\phi^{-1}(2A_i)$, they are $W^{1,2}$-orthogonal, and it is sufficient to prove inequality~\eqref{t1:punch} for all $u_i$, where $i=1,\ldots,k+1$. 

To prove inequality~\eqref{t1:punch} for each $u_i$, we first estimate the numerator in the Rayleigh quotient. Below we assume that $r_i>0$; the case $r_i=0$ can be treated similarly. Since the functions $\psi_i$ and $\bar\psi_i$ are not greater than $1$, we obtain
\begin{equation}
\label{t1:num}
\int_{\Sigma^n}\abs{\nabla u_i}^2d\mathit{Vol}_g\leqslant 2\left(\int_{\phi^{-1}(2A_i)}\abs{\nabla(\psi_i\circ\phi)}^2d\mathit{Vol}_g+\int_{\phi^{-1}(2A_i)}\abs{\nabla(\bar\psi_i\circ\phi)}^2d\mathit{Vol}_g\right).
\end{equation}
Now we claim that the first integral above satisfies the following inequalities
\begin{multline*}
\int_{\phi^{-1}(2A_i)}\abs{\nabla(\psi_i\circ\phi)}^2d\mathit{Vol}_g\leqslant
\int_{\Sigma^n}\abs{\nabla(\psi_i\circ\phi)}^2d\mathit{Vol}_g\leqslant \int_{\Sigma^n}\abs{\nabla(\varphi_{[W_i]}\circ(\theta_{t_i,[W_i]}\circ\phi))}^2d\mathit{Vol}_g\\ \leqslant\int_{\Sigma^n}\abs{\nabla(\tau\circ\theta_{t_i,[W_i]}\circ\phi)}^2d\mathit{Vol}_g,
\end{multline*}
where $\tau$ is the moment map for the action of $\SU_{m+1}$ on the projective space $\mathbb CP^m$. The first relation above is trivial, and the second is the consequence of the definition of the function $\psi_i$, see formula~\eqref{t1:f1}. To explain the last inequality note that, by the equivariance properties,  we may assume that the point $[W_i]$ is $[1:0:\ldots:0]$. Then the function $i\varphi_{[W_i]}$ coincides with the $(1,1)$-entry of the $\mathfrak{su}_{m+1}$-matrix $\tau$, see relation~\eqref{mm:f}, and the inequality follows. Now by Lemma~\ref{denom:top}, we obtain
\begin{multline}
\label{t1:aux}
\int_{\phi^{-1}(2A_i)}\abs{\nabla(\psi_i\circ\phi)}^2d\mathit{Vol}_g\leqslant\int_{\Sigma^n}\abs{\nabla(\tau\circ\theta_{t_i,[W_i]}\circ\phi)}^2d\mathit{Vol}_g= \\ \frac{1}{(n-1)!}\int_{\Sigma^n}(\theta_{t_i,[W_i]}\circ\phi)^*(\omega_{FS})\wedge\omega_g^{n-1} =\frac{1}{(n-1)!}\int_{\Sigma^n}\phi^*(\omega_{FS})\wedge\omega_g^{n-1},
\end{multline}
where in the first equality we used the fact that the volume form on $\Sigma^n$ equals $(\omega_g^n/n!)$, and in the second -- the fact that the pull-back form $\theta_{t_i,[W_i]}^*\omega_{FS}$ is cohomologous to $\omega_{FS}$, see the discussion in Section~\ref{prem}. The second integral in inequality~\eqref{t1:num} can be estimated in a similar fashion. In more detail, by the definition of the function $\bar\psi_i$, see formula~\eqref{t1:f2}, we get
\begin{multline*}
\int_{\phi^{-1}(2A_i)}\abs{\nabla(\bar\psi_i\circ\phi)}^2d\mathit{Vol}_g\leqslant\int_{\Sigma^n}(1+(\varphi_{[W]}\circ(\theta_{\bar t_i,[W_i]}\circ\phi))^{-4}\abs{\nabla(\varphi_{[W_i]}\circ(\theta_{\bar t_i,[W_i]}\circ\phi))}^2d\mathit{Vol}_g\\
\leqslant \int_{\Sigma^n}\abs{\nabla(\varphi_{[W_i]}\circ(\theta_{\bar t_i,[W_i]}\circ\phi))}^2d\mathit{Vol}_g,
\end{multline*}
where in the last inequality we used the fact that the function $\varphi_{[W]}$ is non-negative. Following the line of the argument above, we arrive at the inequality
$$
\int_{\phi^{-1}(2A_i)}\abs{\nabla(\bar\psi_i\circ\phi)}^2d\mathit{Vol}_g\leqslant \frac{1}{(n-1)!}\int_{\Sigma^n}\phi^*(\omega_{FS})\wedge\omega_g^{n-1}.
$$
Combining these estimates for the integrals in the right hand-side of~\eqref{t1:num}, we obtain the following estimate for the Dirichlet integral of $u_i$:
\begin{equation}
\label{t1:num2}
\int_{\Sigma^n}\abs{\nabla u_i}^2d\mathit{Vol}_g\leqslant \frac{4}{(n-1)!}\int_{\Sigma^n}\phi^*(\omega_{FS})\wedge\omega_g^{n-1}.
\end{equation}
Using Lemmas~\ref{t1:l3} and~\ref{t1:l4} together with relation~\eqref{denom1}, we can also estimate the denominator of the Rayleigh quotient:
$$
\int_{\Sigma^n}u_i^2d\mathit{Vol}_g\geqslant \frac{1}{400}\mathit{Vol}_g(\phi^{-1}(A_i))=\frac{1}{400}\mu(A_i)\geqslant\frac{1}{800}c\mu(\mathbb CP^m)/k,
$$
where the constant $c$ depends only on $m$. Recall that the measure $\mu$ above is the push-forward of the volume measure on $\Sigma^n$, and hence, the last inequality gives
\begin{equation}
\label{denom3}
\int_{\Sigma^n}u_i^2d\mathit{Vol}_g\geqslant\frac{1}{800}c\mathit{Vol}_g(\Sigma^n)/k=\frac{c}{800n!}\frac{1}{k}\int_{\Sigma^n}\omega_g^n.
\end{equation}
Combining relations~\eqref{t1:num2} and~\eqref{denom3} we immediately arrive at inequality~\eqref{t1:punch}.
\qed

\section{Laplace eigenvalues of analytic subvarieties}
\label{variety}

\subsection{Laplace operator on analytic subvarieties}
Let $M^{n+l}$ be a complex manifold of dimension $(n+l)$. Recall that an analytic subvariety $\Sigma\subset M^{n+l}$ is a closed subset that is given locally as the zero set of a finite collection of holomorphic functions. A point $p\in\Sigma$ is called regular, if it has a neighbourhood $U$ in $M^{n+l}$ such that $U\cap\Sigma\subset U$ is a complex submanifold. The collection of all regular points is called the {\em regular locus} of $\Sigma$, and is denoted by $\Sigma_*$. An analytic subvariety is called {\em irreducible} if $\Sigma_*$ is connected. Throughout the rest of the section we suppose that $\Sigma^n\subset M^{n+l}$ is an irreducible analytic subvariety whose regular locus has complex dimension $n$. The complement $\Sigma^n\backslash\Sigma_*^n$ is called the {\em singular set} of $\Sigma^n$; it is also a subvariety, but of greater codimension, see~\cite{GH} for details.

For a K\"ahler metric $g$ on $M^{n+l}$ we denote by $g_\Sigma$ the induced metric on a regular locus $\Sigma_*^n$. Let $\Delta_\Sigma$ be the Laplace-Beltrami operator on $\Sigma_*^n$ with respect to the metric $g_\Sigma$. We view $\Delta_\Sigma$ as an operator defined on the set $\mathcal D(\Delta_\Sigma)\subset L^2(\Sigma^n)$ that is formed by compactly supported $C^2$-smooth functions on $\Sigma^n$. The following statement is a version of the result in~\cite{LT95} for algebraic subvarietes. We outline its proof below. Unlike the argument in~\cite{LT95}, we use elliptic regularity theory instead of the integral representaion of the resolvent via the heat kernel.

\begin{prop}
Let $M^{n+l}$ be a closed complex manifold, and $\Sigma^n$ an irreducible analytic subvariety. Then for any K\"ahler metric $g$ on $M^{n+l}$ the Laplace-Beltrami operator $\Delta_\Sigma$ on $\Sigma^n$ is essentially self-adjoint and has discrete spectrum.
\end{prop}
\begin{proof}
The standard Green formula yields the following relation
\begin{equation}
\label{p:r2}
\int_{\Sigma^n}(\Delta_\Sigma u) vd\mathit{Vol}_\Sigma+\int_{\Sigma^n}(du,dv)d\mathit{Vol}_\Sigma=0
\end{equation}
for any $u$, $v\in\mathcal D(\Delta_\Sigma)$. Hence, the operator $\Delta_\Sigma$ is symmetric and has a self-adjoint extension $\bar\Delta$ to an unbounded linear operator on $L^2(\Sigma^n)$, see~\cite[Lemma~1.2.8]{BD}. From~\eqref{p:r2}, we also deduce that
\begin{equation}
\label{p:r3}
\abs{u}_{1,2}^2\leqslant\abs{\langle\Delta_\Sigma u,u\rangle}+\abs{u}^2_2
\end{equation}
for any $u\in\mathcal D(\Delta_\Sigma)$, where $\abs{\,\cdot\,}_{1,2}$ and $\abs{\,\cdot\,}_2$ stand for the $W^{1,2}$-Sobolev norm and $L^2$-norm respectively, and $\langle\cdot,\cdot\rangle$ denotes the $L^2$-scalar product. Using relation~\eqref{p:r3} and elliptic regularity theory, it is straightforward to conclude that the resolvent of $\bar\Delta$ is a bounded linear operator $L^2(\Sigma^n)\to W^{1,2}(\Sigma^n)$. In more detail, let $v$ be a function of the form $(\bar\Delta-\lambda)u$, where $\lambda$ is a point from the resolvent set, and $u\in\mathcal D(\Delta_\Sigma)$. Then inequality~\eqref{p:r3} yields
\begin{equation}
\label{p:r4}
\abs{(\bar\Delta-\lambda)^{-1}v}_{1,2}^2\leqslant\abs{\langle v,(\bar\Delta-\lambda)^{-1}v\rangle}+(\abs{\lambda}+1)\abs{(\bar\Delta-\lambda)^{-1}v}^2_2.
\end{equation}
By elliptic regularity we see that inequality~\eqref{p:r4} holds for an arbitrary compactly supported smooth function $v$. Now, since the set of compactly supported smooth functions is dense among all $L^2$-integrable functions, it is straightforward to see that inequality~\eqref{p:r4} holds for any $L^2$-integrable function $v$.

Thus, the resolvent of $\bar\Delta$ is indeed a bounded linear operator $L^2(\Sigma^n)\to W^{1,2}(\Sigma^n)$. As is shown in the proof of~\cite[Theorem~5.3]{LT95}, the inclusion $W^{1,2}(\Sigma^n)\subset L^2(\Sigma^n)$ is compact, and we conclude that the resolvent of $\bar\Delta$ is also compact. Hence, the operator $\bar\Delta$ has  discrete spectrum. By elliptic regularity the eigenfunctions of $\bar\Delta$ are smooth, and by~\cite[Lemma~1.2.2]{BD} the self-adjoint extension is unique.
\end{proof}

In the argument above we used the statement from the proof of~\cite[Theorem~5.3]{LT95} that the inclusion $W^{1,2}(\Sigma^n)\subset L^2(\Sigma^n)$ is compact. This is the only place where the hypothesis that a metric $g$ on $M^{n+l}$ is K\"ahler is used. The main ingredient in the argument is a version of the Sobolev inequality. In more detail, the subvariety $\Sigma^n$ is a minimal current in $M^{n+l}$, and after an isometric embedding $M^{n+l}\to\mathbb R^m$, is a current of bounded mean curvature in the Euclidean space $\mathbb R^m$ to which the Michael and Simon version of the Sobolev inequality applies, see~\cite{MS73}.

Due to the choice of the domain $\mathcal D(\Delta_\Sigma)$, the self-adjoint extension of $\Delta_\Sigma$ is often referred to as the Dirichlet Laplacian on $\Sigma^n$. Note that in this context the Sobolev space $W^{1,2}_0(\Sigma^n_*)$, the closure in the Sobolev norm of compactly supported smooth functions, coincides with the Sobolev space $W^{1,2}(\Sigma^n)$, see the proof of~\cite[Theorem~4.1]{LT95}, and~\cite[Section~3]{Yosh97}. In particular, the domain $\mathcal D(\Delta_\Sigma)$ is dense in the Sobolev space $W^{1,2}(\Sigma^n)$. We use this observation below for the construction of test-functions for the Laplace eigenvalues on $\Sigma^n_*$.

\subsection{Proof of Theorem~\ref{t2}}
Throughout this section we denote the restriction to the regular locus $\Sigma^n_*$ of a holomorphic map $\phi:M^{n+l}\to\mathbb CP^m$ by the same symbol $\phi$; it is also a holomorphic map. First, by the discussion above the domain $\mathcal D(\Delta_\Sigma)$ is dense in the Sobolev space $W^{1,2}(\Sigma^n)$, and hence, relation~\eqref{p:r2} continues to hold when the function $v$ belongs to $W^{1,2}(\Sigma^n)$. With this observation the standard argument in~\cite{Cha} shows that the variational principle for the eigenvalues $\lambda_k(\Sigma^n,g_\Sigma)$ of the self-adjoint extension continues to hold. Thus, for a proof of the theorem it is sufficient for any integer $k\geqslant 1$ to construct a collection of $W^{1,2}$-orthogonal $k+1$ test-functions $u_i\in W^{1,2}(\Sigma^n)$ such that
$$
\mathcal R(u_i)\leqslant C(n,m)\left(\left(\int_{\Sigma^n}\phi^*(\omega_{FS})\wedge\omega_g^{n-1}\right)/\left(\int_{\Sigma^n}\omega_g^n\right)\right)k,
$$
where $C(n,m)$ is the constant that depends on $n$ and $m$ only. As in the proof of Theorem~\ref{t1}, as test-functions $u_i$ we use the functions  $u_{A_i}\circ\phi$, where the $u_{A_i}$'s are Lipschitz functions on $\mathbb CP^m$ constructed in Section~\ref{proof:t1}. As is known~\cite{GH}, the volume $\mathit{Vol}_g(\Sigma^n_*)$ is finite, and hence, such functions do belong to $W^{1,2}(\Sigma^n)$. 

Now we argue as in the proof of Theorem~\ref{t1}. Let $\mu$ be a measure on $\mathbb CP^m$ obtained by pushing forward the volume measure $\mathit{Vol}_g$ on $\Sigma^n_*$ under a holomorphic map $\phi:\Sigma^n_*\to\mathbb CP^m$. It is finite and has no atoms. The former follows from the fact that the volume $\mathit{Vol}_g(\Sigma^n_*)$ is finite, and the latter -- from the fact that the pre-image $\phi^{-1}([W])$ of any point $[W]\in\mathbb CP^m$ is an analytic subvariety of positive codimension, see~\cite{GH}. Thus, Proposition~\ref{ds1} applies, and we can find a collection $\{A_i\}$ of $k+1$ annuli such that the annuli $\{2A_i\}$ are pair-wise disjoint, and relation~\eqref{denom1} holds. Following the line of the argument in the proof of Theorem~\ref{t1}, we see that estimates~\eqref{t1:num2} and~\eqref{denom3} for the numerator and the denominator respectively in the Rayleigh quotient $\mathcal R(u_i)$, also carry over. In more detail, relation~\eqref{denom3} follows exactly in the same way, and the only point necessary to justify relation~\eqref{t1:num2} is the last equality in~\eqref{t1:aux}. The latter is a consequence of the Stokes formula for analytic varieties, see~\cite{GH}. \qed


\end{document}